\documentclass[11pt,twoside]{amsart}
\usepackage[latin1]{inputenc}
\usepackage[T1]{fontenc}
\usepackage{mathtools}
\usepackage{graphicx}
\usepackage{tikz}
\usetikzlibrary{chains}
\usepackage{tcolorbox}

\usetikzlibrary{matrix}
\usetikzlibrary{shapes.multipart}
\usetikzlibrary{arrows}

\def\inc{37.5}

\PassOptionsToPackage{pdfusetitle,pagebackref,colorlinks}{hyperref}
\usepackage{bookmark}
\hypersetup{
  linkcolor={red!70!black},
  citecolor={green!70!black},
  urlcolor={blue!80!black}
}

\usepackage{amsmath}
\usepackage{amsthm}
\usepackage{amssymb}

\usepackage[all]{xy}
\usepackage{hyperref}
\newtheorem{thm}{Theorem}[section]

\newtheorem{prop}[thm]{Proposition}

\newtheorem{lem}[thm]{Lemma}
\newtheorem{cor}[thm]{Corollary}
\newtheorem{conj}[thm]{Conjecture}

\numberwithin{equation}{section}

\theoremstyle{definition}

\newtheorem{remark}[thm]{Remark}
\newtheorem{ex}[thm]{Examples}
\usepackage[OT2,T1]{fontenc}
\DeclareSymbolFont{cyrletters}{OT2}{wncyr}{m}{n}
\DeclareMathSymbol{\Sha}{\mathalpha}{cyrletters}{"58}

\newcommand{\Gr}{{\rm Gr}}

\newcommand{\nilp}{{\rm nilp}}
\newcommand{\cal}{\mathcal}

\newcommand{\kx}{{\cal X}}

\newcommand{\ZZ}{\mathbb{Z}}
\newcommand{\QQ}{\mathbb{Q}}

\newcommand{\CC}{\mathbb{C}}

\newcommand{\PP}{\mathbb{P}}

\renewcommand{\to}{\xymatrix@1@=15pt{\ar[r]&}}
\renewcommand{\rightarrow}{\xymatrix@1@=15pt{\ar[r]&}}
\renewcommand{\leftarrow}{\xymatrix@1@=15pt{&\ar[l]}}
\renewcommand{\mapsto}{\xymatrix@1@=15pt{\ar@{|->}[r]&}}
\renewcommand{\twoheadrightarrow}{\xymatrix@1@=18pt{\ar@{->>}[r]&}}
\renewcommand{\hookrightarrow}{\xymatrix@1@=15pt{\ar@{^(->}[r]&}}
\newcommand{\hook}{\xymatrix@1@=15pt{\ar@{^(->}[r]&}}
\newcommand{\congpf}{\xymatrix@L=0.6ex@1@=15pt{\ar[r]^-\sim&}}
\renewcommand{\cong}{\simeq}

\begin{document}

\title[]{On type II degenerations of hyperk\"ahler manifolds}
\author[D.\ Huybrechts \& M.\ Mauri]{D.\ Huybrechts \& M.\ Mauri}

\address{Mathematisches Institut and Hausdorff Center for Mathematics, Universit\"at Bonn, Endenicher Allee 60, 53115 Bonn, Germany \& Max Planck Institute for Mathematics, Vivatsgasse 7, 53111 Bonn, Germany}
\email{huybrech@math.uni-bonn.de \& mauri@mpim-bonn.mpg.de}

\begin{abstract} \noindent
We give a simple argument to prove Nagai's conjecture for type II degenerations of compact hyperk\"ahler
manifolds and cohomology classes of middle degree. Under an additional assumption, the techniques
yield the conjecture in arbitrary degree. This would complete the proof of Nagai's conjecture in general,
as it was proved already for type I degenerations by Koll\'ar, Laza, Sacc\`a, and Voisin \cite{KLSV} and independently by Soldatenkov \cite{SoldMoscow}, while it is immediate for type III degenerations. Our arguments are close in spirit to a recent paper by Harder \cite{Harder} proving similar results for 
the restrictive class of \emph{good} degenerations.
 \vspace{-2mm}
\end{abstract}

\maketitle
{\let\thefootnote\relax\footnotetext{The first author is supported by the 
ERC Synergy Grant HyperK. The second author is supported by the Max Planck Institute for Mathematics.}
\marginpar{}
}
\section{Introduction}

Any one-dimensional degeneration $\kx\to\Delta$ of compact K\"ahler manifolds induces monodromy
operators $T_k$ acting on the cohomology groups  $H^k(X,\CC)$ of a smooth fibre $X\coloneqq\kx_t$, ${t\ne0}$. After a base change, one can assume that $T_k$ is unipotent, i.e.\ the operator $T_k-{\rm id}$ is nilpotent.
Alternatively, this can be expressed by saying that the logarithmic monodromy operator
$N_k\coloneqq \log T_k$ acting on $H^k(X,\CC)$ is nilpotent. Thus, the \emph{index of nilpotence}  defined
as $$\nilp(N_k)\coloneqq\max \{ i\mid N^i\ne0\}$$
is finite and it is known  that $\nilp(N_k)\leq k$, see \cite[Ch.\ IV]{Griffiths}. Clearly, if
$\kx\to\Delta$ is a smooth family, then $N_k=0$ for all $k$. In general, there
is no direct link between the $\nilp(N_k)$ for different $k$.
However, following the general philosophy that the geometry of a compact hyperk\"ahler manifold $X$ is largely determined by its Hodge structure of weight two $H^2(X,\ZZ)$,
it was conjectured by Nagai \cite{Nagai} that $N_2$ determines all even $N_{2k}$.

\begin{conj}[Nagai]\label{conj:main}
The even logarithmic monodromy operators
$N_{2k}$, $k\leq n$, of a degeneration of compact hyperk\"ahler manifolds
$\kx\to \Delta$  with unipotent monodromy satisfy
$$\nilp(N_{2k})=k\cdot\nilp(N_2).$$
\end{conj}

Using Verbitsky's result \cite{Bogo,Verb} that cup-product yields  inclusions $S^kH^2(X,\CC)\subset H^{2k}(X,\CC)$ for $k\leq n$, Nagai already observed that 
$$\nilp(N_{2k})\geq k\cdot\nilp(N_2).$$
Thus, for type III degenerations, i.e.\ $\nilp(N_2)=2$, the conjecture follows
from the two inequalities  $2k\geq\nilp(N_{2k})\geq k\cdot\nilp(N_2)$, see also
\cite[Sec.\ 6]{KLSV} or \cite[Sec.\ 5.1]{GKLR}. 

For projective degenerations  of type I, i.e.\  $N_2=0$, 
the conjecture was first established by Koll\'ar, Laza, Sacc\`a, and Voisin \cite{KLSV}.
An independent proof, also valid in the non-projective case, was given by Soldatenkov \cite[Cor.\ 3.6]{SoldMoscow}.
For type II degenerations, i.e.\ $\nilp(N_2)=1$,  it was proved
in \cite[Thm.\ 6.19]{KLSV} that $k\leq \nilp(N_{2k})\leq 2k-2$ for all $2\leq k\leq n-1$, which follows from the observation that the level of the Hodge structure of $H^{2k}(X,\QQ)/S^kH^2(X,\QQ)$ does not exceed $2k-2$.

Thus, in order to establish Nagai's conjecture in full, only the case of type II degenerations 
remains open. In fact, for all known
examples of compact hyperk\"ahler manifolds the conjecture was established by Green, Kim, Laza, and Robles \cite{GKLR}, relying on a complete understanding of the cohomology ring as a representation of the LLV algebra.
Special cases of degenerations of Hilbert schemes and generalized Kummer varieties
have been treated in \cite{Nagai}.
\smallskip

The purpose of this note is to give an elementary proof of the following.

\begin{thm}\label{thm:main}
Let  $\kx\to\Delta$ be a type II degeneration of compact hyperk\"ahler manifolds of dimension $2n$. 
Then $N^{n+1}=0$ on $H^{2n}(X,\CC)$, i.e.\ $${\rm nilp}(N_{2n})=n.$$
Furthermore, ${\rm nilp}(N_{2k})\leq n-1$ for all $k<n$ and, assuming condition \emph{(\ref{eqm:Extraassumption})}, in fact ${\rm nilp}(N_{2k})=k$ for all $k\leq n$.
\end{thm}
\smallskip

Assuming that the degeneration is good, the result for ${\rm nilp}(N_{2n})$ and ${\rm nilp}(N_{2n-2})$ was proved before using different but related arguments by Harder \cite{Harder}. In this sense, our arguments show that the restrictive assumption in \cite{Harder} that the degeneration is good can be dropped.
Good degenerations had earlier been studied by Nagai \cite{Nagai} who observed already that
${\rm nilp}(N_{2k})\not\in\{1,\ldots,k-1\}$. In particular, for good type II degenerations one has
${\rm nilp}(N_{2n-2})=n-1$ which in fact holds for general type II degenerations.
\smallskip

In Section \ref{sec:levelNagai} we establish a link between the perverse and the monodromy
filtration. It leads to the following reformulation of Nagai's conjecture.
\setcounter{section}{3}
\setcounter{thm}{3}
\begin{cor}
The Nagai conjecture (for type II degenerations) holds if and only if the Hodge structure
on the graded pieces  $\Gr^P_{i}H^{2k}(X, \QQ)$ of the perverse filtration has level at most $2k-2|i-k|$. 
\end{cor}


\setcounter{section}{1}
\setcounter{thm}{0}
The techniques that will be used to prove Theorem \ref{thm:main} also yield information 
about the logarithmic monodromy action in odd degree. In Section \ref{sec:Odd} we shall prove the following. 
\setcounter{section}{4}

\begin{thm}
Let  $\kx\to\Delta$ be a type II degeneration of compact hyperk\"ahler manifolds of dimension $2n$.
Then the odd logarithmic monodromy operators $N_{2k-1}$, $k \leq n$, satisfy
$$ \nilp(N_{2k-1})\leq \min\{2k-3, n-1\}.$$ 
Furthermore,  assuming condition \emph{(\ref{eqm:Extraassumption})}, in fact ${\rm nilp}(N_{2k-1})\leq k-1$ for all $k\leq n$.
\end{thm}
There is also a lower bound for ${\rm nilp}(N_{2k-1})$ depending on the level of the Hodge structure,
see Corollary \ref{cor:H3different0}.

\setcounter{section}{1}
\medskip


\section{Even degree}
In the following, $\kx\to \Delta$ will always be a degeneration of compact hyperk\"ahler manifolds with unipotent monodromy and $X=\kx_t$, $t\ne0$, denotes
a fixed smooth fibre.

\subsection{} We begin by recalling some well-known properties of the logarithmic
monodromy operator $N=\bigoplus N_k$.

\begin{lem}\label{lem:Nderivation}
For classes $\alpha\in H^{k}(X,\CC)$ and $\alpha'\in H^\ell(X,\CC)$ one has
\begin{equation}\label{eqn:Nwedge}
N_{k+\ell}(\alpha\wedge\alpha')=N_k(\alpha)\wedge\alpha'+\alpha\wedge N_\ell(\alpha').
\end{equation}
Furthermore, for a degeneration of compact hyperk\"ahler manifolds, the operator $N_2$ is compatible with the Beauville--Bogomolov form: 
\begin{equation}\label{eqn:Nq}
q(N_2(\alpha),\alpha')+q(\alpha,N_2(\alpha'))=0.
\end{equation}
\end{lem}

\begin{proof}
The monodromy operator $T$ on $H^\ast(X,\CC)$ is an algebra isomorphism
and, hence,  \begin{equation}\label{eqn:Tauto}
T^m(\alpha\wedge\alpha')=T^m(\alpha)\wedge T^m(\alpha')
\end{equation} for all $m\geq 0$.
Now use  $T^m=\exp(mN)$ and expand (\ref{eqn:Tauto}) in $m$. Comparing linear terms, yields the first assertion. The proof of the second assertion is similar, relying on the fact that the monodromy operator $T_2$ is orthogonal, i.e.\ $q(T_2(\alpha),T_2(\alpha'))=q(\alpha,\alpha')$.
\end{proof}

Let us now consider a degeneration $\kx\to\Delta$ of compact hyperk\"ahler manifolds with  $T_2-{\rm id}$ of order two
or, equivalently, such that $N_2^2=0$ but $N_2\ne0$,  i.e.\ a degeneration of type II.
For the convenience of the reader we recall the following observation due to Schreieder and Soldatenkov \cite[Prop.\ 4.1]{SchrSold}.

\begin{lem}\label{lem:SchrSold}
For a type II degeneration of compact hyperk\"ahler manifolds the image of the logarithmic monodromy operator $N_2\colon H^2(X,\CC)\to H^2(X,\CC)$
is an isotropic plane.

Furthermore, for any basis $w_1,w_2\in {\rm Im}(N_2)$ 
one has 
\begin{equation}\label{eqn:SchrSold}
N_2(\alpha)=q(w_2,\alpha)\,w_1-q(w_1,\alpha)\,w_2
\end{equation}
up to scaling $N_2$ by a factor.
\end{lem}

\begin{proof}
Due to (\ref{eqn:Nq}), one has $q(N_2(\alpha),N_2(\alpha'))=-q(\alpha,N_2^2(\alpha'))=0$,
which shows that the image is indeed isotropic. A computation with Hodge numbers, see \cite[Ch.\ VI]{Griffiths}, yields $\dim {\rm Im}(N_2)=2$.

Thus, one can write $N_2=\lambda_1(~)\, w_1+\lambda_2(~)\,w_2$
for certain linear forms  $\lambda_1,\lambda_2$ on $H^2(X,\CC)$.
Then, since by (\ref{eqn:Nq}) one has
$q(N_2(\alpha),\alpha)=0$ for all $\alpha\in H^2(X,\CC)$, one finds
$\lambda_1(\alpha)\,q(w_1,\alpha)+\lambda_2(\alpha)\,q(w_2,\alpha)=0$,
which proves $[\lambda_1(~):\lambda_2(~)]=[q(w_2,~):-q(w_1,~)]$.
\end{proof}

\subsection{}\label{sec:betaM} Let us now fix a non-trivial isotropic class
$\beta\in H^{1,1}(X)$ and denote by $L_\beta$ the operator
 $\alpha\mapsto\alpha\wedge\beta$ which is of bidegree $(1,1)$.
 Furthermore, we let  $\Lambda^c_\sigma$ be the operator of
 bidegree $(-2,0)$ defined as the contraction  by the holomorphic
 symplectic form $\sigma$, see also Remark \ref{rem:confusingLambda}.
Then the operator
$$M\coloneqq [L_\beta,\Lambda^c_\sigma]$$
is of bidegree $(-1,1)$ and
enjoys similar properties as the monodromy operator $N$.
 For example, 
 \begin{equation}\label{eqn:Mwedge}M_{k+\ell}(\alpha\wedge \alpha')=M_k(\alpha)\wedge\alpha'+\alpha\wedge M_\ell(\alpha')
 \end{equation} for classes $\alpha\in H^k(X,\CC)$ and $\alpha'\in H^\ell(X,\CC)$.
 Furthermore, $M$ is of type II, i.e.\ $M_2^2=0$, and the image of $M_2$ is an isotropic plane,
 namely ${\rm Im}(M_2)=\langle \beta,\bar\sigma\rangle$. More precisely, up to a scaling factor
 $$M(\sigma)=-q(\bar\sigma,\sigma)\,\beta\text{ and }M(\alpha)=q(\beta,\alpha)\,\bar\sigma$$
 for any $\alpha\in H^{1,1}(X)$. So, similarly to Lemma
 \ref{lem:SchrSold}, we can write,  up to a scaling factor,
 \begin{eqnarray}\label{eqn:Mq}
 M_2=q(\beta,~)\,\bar\sigma-q(\bar\sigma,~)\,\beta.
 \end{eqnarray} 
 The quickest way to prove these facts is via the (complex) Looijenga--Lunts--Verbitsky
 Lie algebra ${\mathfrak g}(X)$, see \cite{LL,Verb} or the summaries  \cite[Thm.\ 2.7]{GKLR} and \cite[Sec.\ 2]{SoldMRL}. It is generated by all $L_x$ and $\Lambda_x$,
 where $x\in H^2(X,\CC)$ can be any class satisfying the Lefschetz property. As $L_\beta$ and $\Lambda_\sigma^c$ are linear combinations of certain $L_x$ and $\Lambda_x$, cf.\ Remark 2.3 below,  
 the known commutator
 relations in ${\mathfrak g}(X)$ also hold for  $[L_\beta,\Lambda^c_\sigma]$
 and the property that ${\mathfrak g}(X)$ acts by derivation on the algebra
 $H^\ast(X,\CC)$ extends to the operator $M$.
 
 \begin{remark}\label{rem:confusingLambda}
(i) If one writes $\sigma=\omega_J+\sqrt{-1}\omega_K$, where $\omega_J,\omega_K$
are the K\"ahler classes of complex structures $J$ and $K$ defined
by the choice of a hyperk\"ahler metric, then $\Lambda^c_\sigma=\Lambda_{\omega_J}-\sqrt{-1}\Lambda_{\omega_K}$. Here, $\Lambda_{\omega_J}$ and $\Lambda_{\omega_K}$ are the usual dual Lefschetz
operators associated to the Lefschetz classes $\omega_J$ and $\omega_K$. In other words, if $\Lambda_\omega$ depending on a Lefschetz class $\omega$ is extended linearly on the whole $H^2(X,\CC)$, then $\Lambda_\sigma^c=\Lambda_{\bar\sigma}$.
For an explanation of the sign we refer to \cite[p.\ 118]{Fujiki}, see also \cite[Sec.\ 3.1]{ShenYin}.
\smallskip

(ii) There exists an isomorphism ${\mathfrak g}(X)\cong \mathfrak{so}(\tilde H(X,\CC))$, see \cite{LL,Verb}. Here, $\tilde H(X,\CC)$ denotes the Mukai extension of $H^2(X,\CC)$, i.e.\ its direct sum with a hyperbolic plane. Furthermore, ${\mathfrak g}(X)\cong{\mathfrak g}_{-2}\oplus\left({\mathfrak g}'\oplus\CC\, h\right)\oplus{\mathfrak g}_2$ where ${\mathfrak g}_{\pm2}\cong H^2(X,\CC)$,
${\mathfrak g}'\cong\mathfrak{so}(H^2(X,\CC))$, and $h$ is the standard counting operator, i.e.\ $h(v)=(d-2n)v$ for $v \in H^{d}(X, \CC)$, 
cf.\ \cite{LL,GKLR,SoldMRL,Verb}.
\end{remark}
 
\subsection{} Nagai's conjecture for the operator $M_{2n}$ can be verified by a computation. This
is the content of the next result. For $k=n$ the assertion will give Nagai's conjecture for the
middle cohomology, see Section \ref{sec:Proofmain}. For $k\leq n/2$ one in fact has the stronger result $M^{2k+1}=0$ on $H^{2k}(X,\CC)$, simply because $M$ is of bidegree $(-1,1)$.

\begin{prop}\label{prop:MNagai}
For all $k\leq n$ we have $M^{n+1}=0$ on $H^{2k}(X,\CC)$. 
\end{prop}

\begin{proof} Since $M$ is of bidegree $(-1,1)$, we have
$M^{\ell+1}=0$ on $H^{p,q}(X)$ for  $p\leq \ell$. Thus, to prove the assertion, it suffices to
show $M^{n+1}(\alpha)=0$ for $\alpha\in H^{p,q}(X)$ with $p+q=2k$ and $p>n$.
 Using the isomorphism $\sigma^{p-n}\colon \Omega_X^{2n-p}\congpf \Omega_X^p$ for $n<p\leq 2n$, we can write $\alpha=\sigma^{p-n}\wedge \gamma$ for some $\gamma\in H^{2n-p,q}$.
Then, by (\ref{eqn:Mwedge}) 
\begin{eqnarray*}
M^{n+1}(\alpha)=M^{n+1}(\sigma^{p-n}\wedge\gamma)&=&\sum_{i+j=n+1}{n+1\choose i}\,M^i(\sigma^{p-n})\wedge M^j(\gamma)\\
&=&
\sum_{\substack{i+j=n+1\\ j\leq 2n-p}}{n+1\choose i}\,M^i(\sigma^{p-n})\wedge M^j(\gamma),
\end{eqnarray*}
where we use that for bidegree reasons $H^{2n-p,q}(X)$ is annihilated by $M^{j}$ for $j>2n-p$.
On the other hand, since $M^2(\sigma)=0$, we have $M^i(\sigma^{p-n})=0$ for $i>p-n$.
This yields the assertion.
\end{proof}

\subsection{} The arguments in the previous proof can be used to deduce $M^{k+1}=0$
on $H^{2k}(X,\CC)$ assuming the following holds for all $p+q\leq 2n-2$:
\begin{equation}\label{eqm:Extraassumption}
\{ \gamma\in H^{p,q}(X)\mid\gamma\wedge\beta=0,\ \gamma\wedge\bar\sigma=0\}=0.
\end{equation}
In fact, it suffices to show triviality for those $\gamma\in H^{2k}(X,\CC)$, $2k<2n$, that are furthermore contained in the image of $M^{k+1}$ or satisfy $\Lambda_{\bar\sigma}(\gamma)\wedge\beta=0$.

\begin{cor}\label{cor:extra}
Assuming \emph{(\ref{eqm:Extraassumption})}, one has
$M^{k+1}=0$ on $H^{2k}(X,\CC)$ for all $2k\leq 2n$.
\end{cor}

\begin{proof}
It suffices to  prove that $M^{k+1}=0$ on $H^{2k}(X,\CC)$, $k<n$, is implied by
$M^{k+2}=0$ on $H^{2k+2}(X,\CC)$, which reduces the assertion to the middle degree
$2n$ covered by Proposition \ref{prop:MNagai}.
 
 Let $\alpha\in H^{p,q}(X)$ with $p+q=2k<2n$ (and $q<k$) and assume that $\ell$ is maximal
with $M^\ell(\alpha)\ne0$. Suppose $\ell\geq k+1$. 
 Since $\ell+1\geq k+2$, $\alpha\wedge w\in H^{2k+2}(X,\CC)$,  and $M^2(w)=0$ for all classes $w\in H^2(X,\CC)$, we have
 $$0=M^{\ell+1}(\alpha\wedge w)=(l+1)M^\ell(\alpha)\wedge M(w).$$
 As the image of $M$ acting on $H^2(X,\CC)$ is spanned by $\bar\sigma$ and $\beta$, assumption (\ref{eqm:Extraassumption}) yields the contradiction $M^\ell(\alpha)=0$.
\end{proof}

\begin{remark}
The condition (\ref{eqm:Extraassumption}) is trivially satisfied for $q<n$, for 
$L_{\bar\sigma}$ is injective on $H^{p,q}(X)$ for $q<n$. Also, if $\beta$ were a K\"ahler class
or simply a class with $q(\beta)\ne0$, then $L_\beta$ would be injective for $2k<2n$.
At this point, we do not know how realistic (\ref{eqm:Extraassumption}) is. Note that, unlike condition (5.2) in
\cite[Thm.\ 5.2]{GKLR},
which is shown to be equivalent to Nagai's conjecture,  condition  (\ref{eqm:Extraassumption})
may fail without contradicting Nagai's conjecture.
\end{remark}

\begin{remark}\label{rem:improv}
(i) Note that the argument in the above proof reducing $M^{k+1}=0$ on $H^{2k}(X,\CC)$ to $M^{k+2}=0$ on $H^{2k+2}(X,\CC)$ works for $2k<n$, for in this case $L_{\bar\sigma}$ is injective on $H^{2k}(X,\CC)$. In other words, without assuming (\ref{eqm:Extraassumption}),
to confirm Nagai's conjecture (for the operator $M$) it suffices to show $M^{k+1}=0$ on $H^{2k}(X,\CC)$ for $n\leq 2k<2n$.

(ii) Also, one can combine the arguments in the proofs of Proposition \ref{prop:MNagai} and Corollary \ref{cor:extra} to show that $M^n=0$ on $H^{2k}(X,\CC)$ for $2k<2n$, because in this case one only needs  
(\ref{eqm:Extraassumption}) for $(0,2k)$-classes which holds true. Indeed, $M^n(\alpha)=0$
for $\alpha\in H^{p,q}(X)$ with $p<n$. For $p\geq n$, write again $\alpha=\sigma^{p-n}\wedge\gamma$
and compute $M^n(\alpha)=\sum_{i+j=n}{n\choose i}\,M^i(\sigma^{p-n})\wedge M^j(\gamma)={n\choose p-n}\,M^{p-n}(\sigma^{p-n})\wedge M^{2n-p}(\gamma)=c\cdot
\beta^{p-n}\wedge M^{2n-p}(\gamma)$ for some constant $c$ with 
$M^{2n-p}(\gamma)\in H^{0,2(k-p+n)}(X)$ and so $M^{2n-p}(\gamma)=c'\cdot \bar\sigma^{k-p+n}$.

Hence, for $w\in H^2(X,\CC)$ with $M(w)=\bar\sigma$ we obtain $0=M^{n+1}(w\wedge\alpha)=c\cdot c'\cdot(n+1)\cdot\beta^{p-n}\wedge\bar\sigma^{k-p+n+1}$.
Since $\bar\sigma^{2n-p}$ induces an isomorphism $H^{p-n,p-n}(X)\simeq H^{p-n,3n-p}(X)$ and since
$k-p+n+1\leq 2n-p$, multiplying classes in $H^{p-n,p-n}(X)$ with $\bar\sigma^{k-p+n+1}$  is injective and, in particular, $\beta^{p-n}\wedge\bar\sigma^{k-p+n+1}\ne0$. This proves $c\cdot c'=0$ and, therefore, $M^n(\alpha)=0$.
\end{remark}

\subsection{}
Let $X$ be a compact hyperk\"ahler manifold and $ H^*(X,\CC)$ its
complex cohomology algebra. In addition, $H^2(X,\CC)$ is endowed with the complex
linear extension of the Beauville--Bogomolov pairing $q$.

Assume $V\subset H^2(X,\CC)$ is an isotropic plane with a basis $v_1,v_2\in V$.
Since $v^{n+1}=0$ in $H^{2n+2}(X,\CC)$ for any isotropic class $v$, the choice of $V=\langle v_1,v_2\rangle$ endows
$H^\ast(X,\CC)$ with the structure of a graded algebra over the Artin algebra
$\CC[x_1,x_2]/(x_1,x_2)^{n+1}$.

The next proposition is a two-parameter version of the observation  \cite[Prop.\ 2.2]{HM}.

\begin{prop}\label{prop:C[x]/x^n+1} 
Any element in the special orthogonal group ${\rm SO}(H^2(X,\CC))$ lifts to
an automorphism of the cohomology ring $H^\ast(X,\CC)$ provided that $b_2\geq 4$. \end{prop}

\begin{proof} We use the shorthand $H^\ast\coloneqq H^\ast(X,\CC)$ and denote by
${\rm Aut}(H^\ast)$  the complex algebraic group of
automorphisms of the graded $\CC$-algebra $H^\ast$. Then, we have to show that the image $G$ of 
the restriction map ${\rm Aut}(H^\ast)\to {\rm GL}(H^2)$ contains ${\rm SO}(H^2,q)$.
For this we use that monodromy defines a discrete subgroup in ${\rm Aut}(H^\ast)$ whose  image 
in ${\rm GL}(H^2)$ contains a finite index subgroup of the integral special orthogonal
group ${\rm SO}(H^2(X,\ZZ))$. 
Since by \cite{Borel} the latter is Zariski dense in ${\rm SO}(H^2,q)$ when $q$ is indefinite, which holds by our
assumption $b_2\geq 4$, we indeed have ${\rm SO}(H^2,q)\subset G$.
\end{proof}

\begin{cor}\label{cor:C[x]/x^n+1}
Assume $b_2\geq 5$, and let $V=\langle v_1,v_2\rangle$ and $V'=\langle v_1',v_2'\rangle$ be two isotropic planes
in $H^2(X,\CC)$. Then the induced graded $\CC[x_1,x_2]/(x_1,x_2)^{n+1}$-algebra structures
on $H^\ast(X,\CC)$ are isomorphic.
\end{cor}

\begin{proof}
If $b_2\geq 5$, any two isotropic planes $V=\langle v_1,v_2\rangle$ and $V'=\langle v_1',v_2'\rangle$ are contained in the same orbit of the action of the complex special orthogonal group ${\rm SO}(H^2,q)$, i.e.\ there exists $g\in{\rm SO}(H^2(X,\CC),q)$ with $g(V)=V'$. Since $g$ lifts to an automorphism of the graded $\CC$-algebra $H^\ast(X,\CC)$, the induced $\CC[x_1,x_2]/(x_1,x_2)^{n+1}$-algebra structures are isomorphic.
\end{proof}

\begin{cor}\label{cor:main}
Let $\kx\to\Delta$ be a type II degeneration of compact
hyperk\"ahler manifolds with $b_2\geq 5$, and let
 $\beta\in H^{1,1}(X)$ be a non-trivial isotropic class  on a smooth fibre $X\coloneqq
 \kx_t$, $t\ne0$. Then the monodromy operator $N$ and the operator $M=[L_\beta,\Lambda^c_\sigma]$, 
 for $0\ne\sigma\in H^{2,0}(X)$, are conjugate under a graded automorphism of the algebra $H^*(X,\CC)$. 
\end{cor}

\begin{proof}
According to (\ref{eqn:SchrSold}) and (\ref{eqn:Mq}) we can write
$N_2=q(w_2,~)\,w_1-q(w_1,~)\,w_2$ for a certain basis $w_1,w_2$ of the
isotropic plane ${\rm Im}(N_2)\subset H^2(X,\CC)$
and, up to a scaling factor, $M_2=q(\beta,~)\,\bar\sigma-q(\bar\sigma,~)\,\beta$.
Now choose a graded algebra automorphism of $H^*(X,\CC)$ that maps
the basis $w_1,w_2$  
of the image of $N_2$ to the basis $\beta,\bar\sigma$ of 
the image of $M_2$. 
Let $N'$ denote the conjugate of the logarithmic monodromy
operator $N$ under this algebra isomorphism. 

Then $N'_2=M_2$ up to a scaling factor and both operators $N'$ and $M$ are contained in the total Lie algebra ${\mathfrak g}(X)$.  For $N'$ this follows from the corresponding assertion for $N$, which is a consequence of \cite[Prop.\ 3.5]{SoldMoscow}, see also \cite[Sec.\ 4]{GKLR}, and for $M=[L_\beta,\Lambda^c_\sigma]$ it holds by definition.
Also, $N'$ and $M$ are both of degree zero and hence contained in the degree zero part of ${\mathfrak g}(X)$. Now use that the degree zero part of ${\mathfrak g}(X)$ is isomorphic to $\mathfrak{so}(H^2(X,\CC))\oplus\CC h$, where $h$ is the degree operator. Hence,  any two degree zero elements of ${\mathfrak g}(X)$ with identical action on $H^2(X,\CC)$ coincide. Thus, the fact that $N'_2=M_2$ immediately
implies $N'=M$.
\end{proof}

\begin{remark}[$b_2=4$]\label{rmk:b2=4}
Note that the existence of an isotropic plane implies $b_2\geq 4$. If $b_2=4$, weaker versions of Corollaries \ref{cor:C[x]/x^n+1} and \ref{cor:main} hold. In this case, the set of isotropic planes is the Fano variety of lines on a quadric in $\PP^3$, which is the disjoint union of two lines. In particular, ${\rm SO}(H^2,q)$ does not act transitively on it. However, the special orthogonal group exchanges any two lines in the same component or, equivalently, whenever
their corresponding isotropic planes are transverse. Therefore, Corollary \ref{cor:C[x]/x^n+1} holds for $b_2=4$ too provided that $V$ and $V'$ are transverse. 

Furthermore, if $b_2=4$, the proof of Corollary \ref{cor:main} does not work, because the isotropic planes $\mathrm{Im}(N_2)$ and $\langle \beta, \bar{\sigma}\rangle$ may not be transverse. Nonetheless, 
we can always replace $\beta$ by a non-trivial isotropic class in $H^{1,1}(X)$ with this property. For instance, if $\beta$ does not work, exchange it with the class $\eta$ to be constructed in the proof of Theorem \ref{thm:Nagai}. We conclude that even for $b_2=4$, the monodromy operator $N$ in Corollary \ref{cor:main} is conjugate to $M =[L_\beta,\Lambda^c_\sigma]$ for some non-trivial isotropic class $\beta \in H^{1,1}(X)$. In particular, we do not need any restriction on Betti number $b_2$ for Theorems \ref{thm:main} and \ref{thm:Nagai}.  
\end{remark}

\subsection{{ \it Proof of Theorem \ref{thm:main}}}\label{sec:Proofmain}

According to Corollary \ref{cor:main} (and Remark \ref{rmk:b2=4}), the monodromy operator $N$ and the operator $M$
correspond to each other under a graded algebra automorphism of $H^\ast(X,\CC)$.
Consider their actions on $H^{2k}(X,\CC)$ for $k\leq n$.  Since $M^{n+1}=0$ on $H^{2k}(X,\CC)$
by Proposition \ref{prop:MNagai}, we can conclude that $N^{n+1}=0$ holds as well.
The stronger statement $N^n=0$ on $H^{2k}(X,\CC)$ for $2k<2n$ follows from  Remark \ref{rem:improv}, (ii).

Furthermore, assuming (\ref{eqm:Extraassumption}), we deduce from Corollary \ref{cor:extra} that
$M^{k+1}=0$ on $H^{2k}(X,\CC)$, which again by  Corollary \ref{cor:main} implies the assertion $N^{k+1}=0$.
\qed

\section{Nagai's conjecture and the perverse filtration}\label{sec:levelNagai}
 Let $X$ be a compact hyperk\"ahler manifold of complex dimension $2n$. In this section we show that, quite remarkably, Nagai's conjecture prescribes the shape of the perverse filtration associated to a Lagrangian fibration $f\colon X \to B$. 
 
 \subsection{}\label{defn:filtration} We start by recalling the filtration associated to a nilpotent operator,  see \cite[\S 1.6]{Deligne}.

 Given a nilpotent endomorphism $N$ of a finite dimensional vector space $V$ of index of nilpotence
 $k$, i.e.\ $N^{k}\neq 0$ and $N^{k+1}= 0$, the \emph{weight filtration of $N$ centred at $k$} is the unique increasing filtration
  $W^N_0V\subset W^N_1V\subset\cdots \subset W^N_{2k-1}V\subset W^N_{2k}V=V$,
   such that $$N W^N_i \subseteq W^N_{i-2}\text{ and }N^i \colon \Gr^N_{k+i} V \congpf \Gr^N_{k-i}V.$$
Here, we use $N$ to denote also the induced maps on the graded pieces $\Gr^N_{i} V \coloneqq W^N_iV/W^N_{i-1}V$.

\smallskip
  
 \begin{ex} Let us consider some geometric examples.
  \smallskip

{\rm   (i)} For  $0\ne \bar{\sigma}\in H^{0,2}(X)$, the weight filtration of the nilpotent operator $L_{\bar{\sigma}}$ on $H^*(X, \CC)$ centred at $n$ is the conjugate Hodge filtration
  , i.e.\
  $W^{L_{\bar{\sigma}}}_{i}H^*(X, \CC)= \sum_{q\geq 2n- i}H^{p,q}(X)$, see \cite[Prop.\ 2.6]{Fujiki}.  
  \smallskip
  
 
{\rm   (ii)} Assume $f\colon X\to B$  is a Lagrangian fibration.
  Up to a scalar factor, there exists a unique class $\beta\in H^2(X, \QQ)$ which is the pullback of an ample class in $H^2(B, \QQ)$, see \cite{Matsushita}. The weight filtration of $L_\beta$  on $H^*(X, \QQ)$ centred at $n$ is the perverse filtration
  $$
  P_i H^d(X, \QQ)\coloneqq\sum_{j\geq 0}
  \beta^j \cdot {\rm Ker}\left(\beta^{n-(d-2j)+i+1}\colon H^{d-2j}(X, \QQ)\to H^{2n-d+2j+2i+2}(X, \QQ)\right),$$
   i.e.\ $W^{L_\beta}_i H^*(X, \QQ) \cap H^d(X, \QQ)=P_{d+i-2n}H^{d}(X, \QQ)$, see  \cite[Thm.\ 2.1.5]{dCM}.
  In particular, the graded pieces $\Gr^P_{i}H^d(X, \QQ)$
  are pure Hodge structures of weight $d$.
  
  \smallskip
  
{\rm (iii)} Let $\kx\to\Delta$ be a type II degeneration of compact
hyperk\"ahler manifolds with unipotent monodromy and $X=\kx_t$, $t\ne0$, 
a fixed smooth fibre. Up to a shift, the weight filtration  of the logarithmic monodromy operator $N$ on $H^*(X, \CC)$ centred at $n$ coincides with the monodromy filtration $W_{\rm mon}$ 
 defined in \cite[Sec.\ 11.2.5]{PS}. On the other hand, by Corollary \ref{cor:main}, the weight filtration of $N$ is conjugate to the weight filtration of $M=[L_{\beta}, \Lambda^c_{\sigma}]$. Together this becomes
  \[\Gr^{W_{\rm mon}}_{k+j}H^k(X, \CC) \simeq \Gr^N_{n+j}H^k(X, \CC)\simeq \Gr^M_{n+j}H^k(X, \CC).\] 
  Observe that the dimension of $\Gr^N_{n+j}H^k(X, \CC)$ is independent of the choice of $\kx\to\Delta$ by Corollary \ref{cor:main}, see \cite{KLSV, SoldMoscow} for the analogous statements in type I and III.
  \end{ex}
 

\subsection{} The  goal of this section is to provide an equivalent formulation of Nagai's conjecture, see
\cite[Prop.\ 1.15]{GKLR} for another one using the representation theory of the LLV algebra. We begin by
establishing a link between the monodromy filtration and the perverse filtration.

As before, $X$ denotes a compact hyperk\"ahler manifold of complex dimension $2n$.

\begin{thm}\label{thm:Nagai}
Let $N$ be the logarithmic monodromy operator on $H^*(X, \CC)$ of a type II degenerations
and assume that there exists a Lagrangian fibration $f\colon X \to B$. For the two induced filtrations there exists an isomorphism
 \begin{equation}\label{eq:MP}
\Gr^N_{n+j}H^{\ell}(X, \CC)\simeq \bigoplus_{p+q=\ell}\Gr^P_{j+q}H^{p,q}(X).
\end{equation}
\end{thm}

\begin{proof}
Let $(g, I, J, K)$ be a hyperk\"{a}hler structure on the complex manifold $X$, i.e.\ a Riemannian metric $g$ which is K\"{a}hler with respect to three complex structures $I$ (the one defining $X$), $J$, and $K$  satisfying the quaternion relations $I^2=J^2=K^2=IJK=-1$. 0 Assume $q(\omega_I)=q(\omega_J)=q(\omega_K)=2$ and
let $\beta = f^*\alpha$ be the pullback of an ample class $\alpha \in H^2(B, \QQ)$ normalized such that $q(\omega_I,\beta)=1$.  
The class $\eta \coloneqq \omega_I - \beta$ is isotropic, i.e.\ $q(\eta)=0$. 
Note also that $q(\beta, \eta)=q(\beta, \omega_I)=q(\bar{\sigma}, \sigma)=1$. We shall denote by $U\subset H^2(X,\CC)$ the orthogonal complement of $\langle \beta, \omega_I, \omega_J, \omega_K \rangle$.

By Proposition \ref{prop:C[x]/x^n+1}, Corollary \ref{cor:C[x]/x^n+1} and Remark \ref{rmk:b2=4}, there exists an automorphism $\varphi \in \operatorname{Aut}(H^*(X, \CC))$ of the graded algebra $H^*(X, \CC)$ which is an involution on $H^2(X,\CC)$ with $\varphi(\sigma)=\beta$, $\varphi(\eta)= \bar\sigma$, and $\varphi|_U={\rm id}$. Let $y\coloneqq \varphi(\omega_J)$, $z\coloneqq \varphi(\omega_K) \in H^2(X,\CC)$. 
Then the following are all $\mathfrak{sl}_2$-triples:
{\small $$E_{\sigma}\coloneqq L_{\sigma} = \frac{1}{2}L_{\omega_J}+\frac{\sqrt{-1}}{2}L_{\omega_K},~ F_{\sigma}\coloneqq \Lambda_\sigma^c=\Lambda_{\bar\sigma}= \frac{1}{2}\Lambda_{\omega_J}-\frac{\sqrt{-1}}{2}\Lambda_{\omega_K},~H_{\sigma}=(p-n)\,{\rm id} \text{ on } H^{p,q}(X),$$}
{\small $$E_{\bar{\sigma}}\coloneqq L_{\bar{\sigma}} = \frac{1}{2} L_{\omega_J}-\frac{\sqrt{-1}}{2}L_{\omega_K},~F_{\bar{\sigma}}\coloneqq \Lambda^c_{\bar\sigma}= \Lambda_{\sigma} = \frac{1}{2}\Lambda_{\omega_J}+\frac{\sqrt{-1}}{2}\Lambda_{\omega_K},~H_{\bar{\sigma}}=(q-n)\, {\rm id} \text{ on } H^{p,q}(X),$$}
{\small $$E_{\beta}\coloneqq \varphi E_{{\sigma}} \varphi^{-1}=L_{\beta},~ F_{\beta}\coloneqq \varphi F_{{\sigma}} \varphi^{-1}=\Lambda_\eta= \frac{1}{2}\Lambda_{y}-\frac{\sqrt{-1}}{2}\Lambda_{z},~ H_{\beta}=[E_{\beta}, F_{\beta}]=\varphi H_{\sigma} \varphi^{-1},$$}
and
{ $$ E_{\eta}\coloneqq \varphi E_{\bar\sigma} \varphi^{-1}= L_{\eta},~ F_{\eta}\coloneqq \varphi F_{\bar\sigma} \varphi^{-1}=\Lambda_\beta,~H_{\eta}=\varphi H_{\bar{\sigma}} \varphi^{-1}.$$}
Since $H_\beta$ is a morphism of Hodge structures, we obtain the decomposition 
\begin{equation}\label{eq:decomp}
H^{p,q}(X) = \bigoplus_{i} V^{p,q,i}
\end{equation}
such that 
$H_{\beta}=(p+q-i-n)\,{\rm id}$ on $V^{p,q,i}$. In particular, the representation theory of $\mathfrak{sl}_2$-triples gives $ V^{p,q,i} \simeq \Gr^P_i H^{p,q}(X)$. Next one observes that  
{\small\begin{equation}\label{eq:MMM} E_{M}\coloneqq 2M = 2[E_{\beta}, F_{{\sigma}}]=2[L_\beta,\Lambda_\sigma^c],~ F_{M}\coloneqq 2[F_{\bar\sigma},E_{\eta}]=2[\Lambda_{\bar\sigma}^c, L_\eta],\\~
H_{M}\coloneqq[E_{M}, F_{M}]=H_{\beta}-H_{\sigma}
\end{equation}}
is an $\mathfrak{sl}_2$-triple. Indeed, by \cite[Lem.\ 3.9]{KSV} or \cite[Thm.\ 2.7]{GKLR}, see also Section \ref{sec:betaM},
\[E_{M}|_{H^2(X, \CC)}=q(\beta,\,)\,\bar{\sigma}-q(\bar{\sigma},\,)\,\beta\text{ and } F_{M}|_{H^2(X, \CC)}=q(\sigma,\,)\,\eta-q(\eta,\,)\, \sigma.\]
From here it is straightforward to check that $(E_M, F_M,H_M)$ is an $\mathfrak{sl}_2$-triple in the Lie algebra $\mathfrak{so}(H^2(X, \CC))$ and, hence, also in the LLV algebra ${\mathfrak g}(X)\cong\mathfrak{so}(\tilde H(X,\CC))$.

 Finally by \eqref{eq:MMM}, the eigenspaces of $H_{M}$ are $\bigoplus_{p+q=\ell}V^{p,q,j+q} \simeq \Gr^M_{n+j} H^{\ell}(X,\CC)$, and hence
 \begin{equation*}
\Gr^N_{n+j}H^{\ell}(X, \CC) \simeq \Gr^M_{n+j}H^{\ell}(X, \CC)\simeq \bigoplus_{p+q=\ell}\Gr^P_{j+q}H^{p,q}(X),
\end{equation*}
which concludes the proof of the theorem.
\end{proof}

  \begin{remark}\label{rmk:generalbeta}
  Despite being evocative from a geometric viewpoint, it is not necessary to assume the existence of a Lagrangian fibration in Theorem  \ref{thm:Nagai}. Indeed, in the above proofs we can replace $\beta$ with any non-zero isotropic class of type $(1,1)$ by \cite[Prop.\ 2.2]{HM}.
  \end{remark}

The theorem immediately yields an equivalent reformulation of Nagai's conjecture in terms of the perverse filtration.

\begin{cor}\label{cor:Nagaiconjperverse}
The Nagai conjecture (for type II degenerations)
holds if and only if the Hodge structure $\Gr^P_{i}H^{2k}(X, \QQ)$ has level at most $2k-2|i-k|$. 
\end{cor}
\begin{proof}
The Nagai conjecture for $N$ says that $\nilp(N_{2k})=k$, i.e.\
$\Gr^N_{n+j}H^{2k}(X, \CC)=0$ for $|j|>k$, see Section \ref{defn:filtration}. By Theorem \ref{thm:Nagai} this condition is equivalent to $\Gr^P_{i}H^{p,q}(X)=0$ for $|i-q|>k$, i.e.\ $\Gr^P_{i}H^{2k}(X, \QQ)$ has level at most $2k-2|i-k|$.
\end{proof}

The $(q,i)$-entry in the following picture is the direct summand $V^{2-q,q,i}$ of $H^2(X, \CC)$.
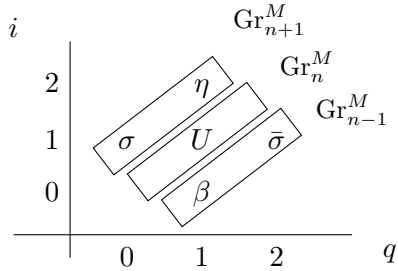
\begin{figure}[h]
      \centering
    \begin{tikzpicture}
    \matrix (m) [matrix of math nodes,
      nodes in empty cells,nodes={minimum width=5ex,
      minimum height=1ex,outer sep=-5pt},
      column sep=1ex,row sep=1ex]{ 
                  & &  &    & \\
            2     & & \eta &    & \\
            1     & \sigma & U  & \bar{\sigma} & \\
            0     & &   \beta &    & \\
      \quad\strut & 0 & 1 & 2  &  \strut \\};
  \draw (m-1-1.east) -- (m-5-1.east) ;
  \draw (-2.5,-1) -- (2,-1);
  \node[below] at (2.5,-1) {$q$};
  \node[above] at (-2.5,1.5) {$i$};
  \node[above] at (0.9, 1.5) {\small $\Gr^M_{n+1}$};
  \node[above] at (1.4, 0.9) {\small $\Gr^M_n$};
  \node[above] at (2, 0.3) {\small $\Gr^M_{n-1}$};
  \node[above] at (-0.5, -3.5) {};
  \begin{scope}[rotate=180+\inc]
      \draw (-1.25, 0.55) rectangle ++ (2,-0.45);
      \draw (-1.1, -0) rectangle ++ (2,-0.45);
      \draw (-0.95,-0.55) rectangle ++ (2,-0.45);
    \end{scope}
  \end{tikzpicture}
  \vspace{-1.9 cm}
    \label{fig:my_label}
  \end{figure}

 The sums along the columns give the Hodge decomposition of $H^2(X, \QQ)$; the sums along the rows split the perverse filtration on $H^2(X, \QQ)$; the sums along the northeast-southwest diagonals split the weight filtration of $M$. Complex conjugation and the conjugate via $\varphi$ of complex conjugation account for the symmetries of this diamond, respectively the reflection about the vertical and horizontal middle axes.
 The Nagai conjecture predicts the existence of these diamonds in any cohomological degree.
\smallskip

The arguments above also yield the following observation concerning the
cohomology of odd degree of a compact hyperk\"ahler manifold, see also \cite[Lem.\ 1.2]{Fujiki} or \cite[Cor.\ 8.1]{Wak58} where the hypothesis $b_{2}\geq 4$ is not needed.

\begin{cor}
If $b_{2}\geq 4$, then all odd Betti numbers $b_{2k-1}$ are divisible by $4$. 
\end{cor}

\begin{proof} If $b_{2}\geq 4$, there exists a non-zero isotropic class $\beta$ of type $(1,1)$. Consider then the involution $\varphi$ in the proof of Theorem \ref{thm:Nagai} and the decomposition \eqref{eq:decomp}, both well-defined by Remark \ref{rmk:generalbeta}.
The map $\varphi$ exchanges the eigenvectors of $H_{\bar{\sigma}}$ and $H_{\beta}$, i.e.\
$\varphi(V^{p,q,i})=V^{i,p+q-i,p}$. 
By complex conjugation, we obtain
\begin{equation}\label{eq:symmetries}
V^{q,p,i}\simeq V^{p,q,i}\simeq V^{i,p+q-i,p} \simeq  V^{p+q-i, i,p}.
\end{equation}
We conclude that
\[
H^{2k-1}(X, \CC) = \bigoplus_{p+q=2k-1}  V^{p,q,i}\simeq \bigoplus_{p+q=2k-1, p< k, i < k}  (V^{p,q,i})^{\oplus 4},
\]
which proves the assertion.
\end{proof}


\section{Odd degree}\label{sec:Odd}

We shall now discuss the action of the logarithmic monodromy on the odd degree cohomology. 
Again, $\kx\to\Delta$ is a type II degeneration of compact hyperk\"ahler manifolds of dimension $2n$. 

\begin{thm}\label{thm:mainodd}
The odd logarithmic monodromy operators $N_{2k-1}$, $k \leq n$, satisfy
$$ \nilp(N_{2k-1})\leq \min\{2k-3, n-1\}.$$ 
Furthermore, assuming condition \emph{(\ref{eqm:Extraassumption})}, in fact ${\rm nilp}(N_{2k-1})\leq k-1$ for all $k\leq n$.
\end{thm}

For any degeneration types one has  $\nilp(N_{2k-1})\leq 2k -3$,
 as the level of the Hodge structure $H^{2k-1}(X,\CC)$ is at most $2k-3$.   For type III degenerations, i.e.\ $\nilp(N_2)=2$, 
 Soldatenkov  proves \cite[Prop.\ 3.15]{SoldMoscow} that the upper bound $2k-3$ is
attained if $H^{3}(X,\CC)\ne0$. In fact, one can show that for type III degenerations
$\nilp(N_{d})$ is always the level of the Hodge structure $H^d(X,\CC)$, see \cite[Prop.\ 3.12]{HM}. For type I degenerations all logarithmic monodromy
operators are trivial, cf.\  \cite{KLSV}.


\subsection{\it Proof of Theorem \ref{thm:mainodd}}
As we argue in Section \ref{sec:Proofmain}, it suffices to show the analogous statements for $M_{2k-1}$. 
Since $M$ is of bidegree $(-1, 1)$ and the Hodge structure of $H^{2k-1}(X, \QQ)$ has level at most $2k-3$, we have $\nilp(M_{2k-1})\leq 2k-3$. 

To show that $\nilp(M_{2k-1})\leq n-1$, we imitate the arguments in the proof of Proposition \ref{prop:MNagai}
to show $M^n=0$ on $H^{2k-1}(X,\CC)$. 
Assume  $\alpha\in H^{p,q}(X)$ with $p+q=2k-1$. If $p< n$, then clearly $M^n(\alpha)=0$
and for $p=n$, we have $M^n(\alpha)\in H^{0,2k-1}(X)=0$.
For $p> n$, we write again $\alpha=\sigma^{p-n}\wedge\gamma$
for some $\gamma\in H^{2n-p,q}(X)$. Then use $M^i(\sigma^{p-n})=0$
for $i>p-n$ and $M^j(\gamma)=0$ for $j\geq 2n-p$ to conclude $M^n(\alpha)=0$.\qed

Note that the induction in the proof of Corollary \ref{cor:extra} under the assumption (\ref{eqm:Extraassumption}) will work here as well. So, assuming (\ref{eqm:Extraassumption}), one can actually show ${\rm nilp}(N_{2k-1})\leq k-1$ for
all $k\leq n$.

\subsection{} Let us now turn to lower bounds, which can not be obtained directly from ${\rm nilp}(N_2)$.
\begin{cor}\label{cor:H3different0} Assume
that the level of the Hodge structure $H^{2k-1}(X, \QQ)$, $k\leq n$, is $2\ell-1>0$. Then
$$\ell \le{\rm nilp}(N_{2k-1}).$$
If $H^3(X, \QQ) \neq 0$, then $\ell=k-1$ and, therefore, 
$$k-1\leq \nilp(N_{2k-1}).$$
In particular, $\nilp(N_3)=1$ and $\nilp(N_{2n-1})=n-1$. 
\end{cor}
\begin{proof}
We prove first that $\ell\leq \nilp(M_{2k-1})$. By assumption, $H^{k-\ell,k+\ell-1}(X) \neq 0$, and so there exists an integer $i$ such that $V^{k-\ell,k+\ell-1, i} \neq 0$. Then, by \eqref{eq:symmetries} and \eqref{eq:MP}, $\Gr^M_{n+j}H^{2k-1}(X, \CC) \neq 0$ if $j = |k-\ell-i|$ and $|k+\ell-1-i|$. 
Therefore, we obtain
\[\ell= \min_i  \max\{|k-\ell-i|, |k+\ell-1-i|\}\leq \nilp(M_{2k-1}).
\]

If $H^3(X, \QQ) \neq 0$, the level of the Hodge structure $H^{2k-1}(X, \QQ)$, for $k \leq n$, is exactly $2(k-1)-1$. Indeed, since $\bar{\sigma}^{n-2}$ induces an isomorphism $H^{1,2}(X) \simeq H^{1,2n-2}(X)$, the map $\bar{\sigma}^{k-2}\colon H^{1,2}(X)\to H^{1, 2k-2}(X)$ is injective, and so $H^{1,2k-2}(X) \neq 0$.
\end{proof}

\section*{Acknowledgements} 
\noindent We wish to thank Oliver Debarre for the reference \cite{Borel}, Andrey Soldatenkov for comments on a first version of this note, and an anonymous referee for helpful remarks.

\bibliography{HyperHM}
\bibliographystyle{mrl}
\bibliographystyle{plain}
\end{document}